\documentclass[12pt,a4paper]{article}
\usepackage[utf8]{inputenc}
\usepackage[affil-it]{authblk}
\usepackage{amsmath}
\usepackage{amsthm}
\usepackage{bbm}
\usepackage{bm}
\usepackage[top=3cm,left=3cm,right=3cm,bottom=3cm]{geometry}
\usepackage{txfonts}
\usepackage{amsfonts}
\usepackage{amssymb}
\usepackage{hyperref}
\usepackage{enumerate}
\usepackage{enumitem}

\usepackage{graphicx}
\usepackage{tikz-cd}
\usetikzlibrary{decorations.markings}
\usetikzlibrary{shapes,positioning,intersections,quotes}

\usepackage{color}
\usepackage{url}

\newcommand{\p}{\partial}
\newcommand{\id}{\textrm{id}}

\newcommand{\NN}{\mathbb N}
\newcommand{\IN}{\mathbb Z}
\newcommand{\RN}{\mathbb R}
\newcommand{\CN}{\mathbb C}

\newcommand{\TT}{\mathbb T}
\newcommand{\SP}{{\mathbb S}^1}

\allowdisplaybreaks[1]

\makeindex

\numberwithin{equation}{section}

\begin{document}


\theoremstyle{plain}
\newtheorem{thm}{Theorem}[section]
\newtheorem{lem}[thm]{Lemma}
\newtheorem{cor}[thm]{Corollary}
\newtheorem{prop}[thm]{Proposition}
\newtheorem{conj}[thm]{Conjecture}

\theoremstyle{definition}
\newtheorem{defn}[thm]{Definition}

\theoremstyle{remark}
\newtheorem{rmk}[thm]{Remark}
\newtheorem{exam}[thm]{Example}


\title{Escaping orbits are also rare in the almost periodic Fermi-Ulam ping-pong}
\author{Henrik Schließauf
	\thanks{
		Electronic address: \texttt{hschlies@math.uni-koeln.de}
	}
}
\affil{
	Universität zu Köln, Mathematisches Institut, \\
	Weyertal 86-90, 50931 Köln, Germany
}
\date{August 6, 2019}

\maketitle

\begin{abstract}
	\noindent 
	We study the one-dimensional Fermi-Ulam ping-pong problem with a Bohr almost periodic forcing function and show that the set of initial condition leading to escaping orbits typically has Lebesgue measure zero.
\end{abstract}

\section{Introduction}

The Fermi-Ulam ping-pong is a model describing how charged particles bounce off magnetic mirrors and thus gain energy. They undergo the so called Fermi acceleration and one central question is whether the particles velocities can get close to the speed of light that way. The model was introduced by Fermi \cite{Fermi_on_the_origin} in order to explain the origin of high energy cosmic radiation. A common one-dimensional mathematical formulation of this problem is as follows: The point particle bounces completely elastically between two vertical plates of infinite mass, one fixed at $x=0$ and one moving in time as $x=p(t)$ for some forcing function $p=p(t)>0$. The particle alternately hits the walls and experiences no external force in between the collisions. The motion can be described by the successor map $f:(t_0,v_0) \mapsto (t_1,v_1)$, mapping the time $t_0\in \RN$ of an impact at the left plate $x=0$ and the corresponding velocity $v_0>0$ right after the collision to $(t_1,v_1)$, representing the subsequent impact at $x=0$. Since one is interested in the long term behavior, we study the forward iterates $(t_n,v_n) = f^n(t_0,v_0)$ for $n \in \NN$ and in particular the `escaping set'
\begin{equation*}
E = \{ (t_0,v_0) : \lim_{n\to\infty} v_n = \infty \},
\end{equation*}
consisting of initial data, which lead to infinitely fast particles. The most studied case is that of a periodic forcing $p(t)$. Ulam \cite{ulam1961} conjectured an increase in energy with time on the average. Based on some numerical simulations, he however realized that rather large fluctuations and no clear gain in energy seemed to be the typical behavior. Two decades later, the development of KAM theory allowed to prove that the conjecture is indeed false. If the forcing $p$ is sufficiently smooth, all orbits stay bounded in the phase space, since the existence of invariant curves prevents the orbits from escaping \cite{Laederich_1991,Pustylnikov_1983}. The proofs are based on Moser's twist thoerem \cite{Moser}, which relies on the higher regularity. And indeed, Zharnitsky \cite{Zharnitsky_1998} showed the existence of escaping orbits if only continuity is imposed on $p$. In the non-periodic case, one can even find $\mathcal{C}^\infty$-forcings with this behavior \cite{kunze_ortega_complete_orbits}. More recently, Dolgopyat and De Simoi developed a new approach. They consider the periodic case and study some maps which are basically approximations of the successor map $f$. This way they could prove several results regarding the Lebesgue measure of the escaping set $E$ \cite{dolgopyat,Dolgopyat_2008,de_Simoi_dolgo,De_Simoi_2013}. \\
Finally, Zharnitsky \cite{Zharnitsky_invariant_curve} investigated the case of a quasi-periodic forcing function whose frequencies satisfy a Diophantine inequality. Again, using an invariant curve theorem, he was able to show that the velocity of every particle is uniformly bounded in time. Since no such theorem is available if the Diophantine condition is dropped, a different approach is necessary in this case. This was done by Kunze and Ortega in \cite{kunze_ortega_ping_pong}. They apply a refined version of the Poincaré recurrence theorem due to Dolgopyat \cite{Dolgopyat_lectures} to the set of initial condition leading to unbounded orbits, and thereby show that most orbits are recurrent. Thus, typically the escaping set $E$ will have Lebesgue measure zero. Now, in this work we will give an affirmative answer to the question raised in \cite{kunze_ortega_ping_pong} whether this result can be generalized to the almost periodic case. Indeed, most of their arguments translate naturally into the language of Bohr almost periodic functions. Our main theorem (Theorem \ref{thm main theorem almost periodic}) states that the escaping set $E$ is most likely to have measure zero, provided the almost periodic forcing $p$ is sufficiently smooth. \\
In order to explain more precisely what we mean by `most likely', we first need to introduce some properties and notation regarding almost periodic functions. This is done in section \ref{section a p functions}. Subsequently we will study measure-preserving successor maps of a certain type and their iterations. We end this part by stating Theorem \ref{thm escaping set nullmenge}, a slightly generalized version of a theorem by Kunze and Ortega \cite{kunze_ortega_ping_pong}, which describes conditions under which the escaping set typically will have measure zero. This will be the most important tool and its proof will be given in the following section. Then, in the last section we discuss the ping-pong model in more detail and finally state and prove the main theorem.

\section{Almost periodic functions and their representation} \label{section a p functions}
\subsection{Compact topological groups and minimal flows} \label{sec compact top groups}

Let $\Omega$ be a commutative topological group, which is metrizable and compact. We will consider the group operation to be additive. Moreover, suppose there is a continuous homomorphism $\psi:\RN \to \Omega$, such that the image $\psi(\RN)$ is dense in $\Omega$. This function $\psi$ induces a canonical flow on $\Omega$, namely 
\begin{equation*}
\Omega\times\RN\to\Omega, \;\;  \omega \cdot t = \omega + \psi(t).
\end{equation*}
This flow is minimal, since
\begin{equation*}
\overline{\omega \cdot \RN} = \overline{\omega + \psi(\RN)} = \omega + \overline{\psi(\RN)} = \Omega
\end{equation*}
holds for every $\omega \in \Omega$.  Let us also note that in general $\psi$ can be nontrivial and periodic, but this happens if and only if $\Omega \cong \SP$ \cite{Ortega_Tarallo}.\\
Now consider the unit circle $\SP = \{ z \in \CN : \lvert z \rvert = 1\}$ and a continuous homomorphism $\varphi:\Omega\to \SP$. Such functions $\varphi$ are called characters and together with the point wise  product they form a group, the so called dual group $\Omega^*$. Its trivial element is the constant map with value $1$. It is a well known fact that nontrivial characters exist, whenever $\Omega$ is nontrivial \cite{Pontryagin_1966}. Also non-compact groups admit a dual group. Crucial to us will be the fact that
\begin{equation*}
\RN^* = \{ t \mapsto e^{i\alpha t}: \alpha \in \RN \}.
\end{equation*}
Now, for a nontrivial character $\varphi \in \Omega^*$ we define
\begin{equation*}
\Sigma = \ker \varphi = \{ \omega \in \Omega : \varphi(\omega) = 1 \}.
\end{equation*}
Then $\Sigma$ is a compact subgroup of $\Omega$. If in addition $\Omega \ncong \SP$, it can be shown that $\Sigma$ is perfect \cite{Ortega_Tarallo}. This subgroup will act as a global cross section to the flow on $\Omega$. Concerning this, note that since $\varphi \circ \psi$ describes a nontrivial character of $\RN$, there is a unique $\alpha \neq 0$ such that
\begin{equation*}
\varphi(\psi(t))=e^{i \alpha t }
\end{equation*}
for all $t \in \RN$. Therefore, the minimal period of this function,
\begin{equation*}
S = \frac{2\pi}{\lvert \alpha \rvert},
\end{equation*}
can be seen as a returning time on $\Sigma$ in the following sense. If we denote by $\tau(\omega)$ the unique number in $[0,S)$ such that $\varphi(\omega) = e^{i\alpha \tau(\omega)}$, then one has
\begin{equation*}
\varphi(\omega\cdot t)=\varphi(\omega+\psi(t)) = \varphi(\omega)\varphi(\psi(t))= e^{i \alpha \tau(\omega)} e^{i \alpha t}
\end{equation*}
and thus
\begin{equation*}
\omega \cdot t \in \Sigma \Leftrightarrow t \in -\tau(\omega) + S\IN.
\end{equation*}
Also $\tau$ as defined above is a function $\tau:\Omega \to [0,S)$ that is continuous where $\tau(\omega)\neq0$, i.e.\ on $\Omega \setminus \Sigma$.
From this we can derive that the restricted flow
\begin{equation*}
\Phi : \Sigma \times [0,S) \to \Omega, \;\; \Phi(\sigma,t) = \sigma \cdot t,
\end{equation*}
is a continuous bijection. Like $\tau(\omega)$, its inverse
\begin{equation*}
\Phi^{-1}(\omega) = (\omega \cdot (-\tau(\omega)), \tau(\omega))
\end{equation*}
is continuous only on $\Omega\setminus \Sigma$. Therefore, $\Phi$ describes a homeomorphism from $\Sigma \times (0,S)$ to $\Omega \setminus \Sigma$.

\begin{exam} \label{exam torus}
	One important example for such a group $\Omega$ is the $N$-Torus $\TT^N$, where $\TT = \RN / \IN$. We will denote classes in $\TT^N$ by $\bar\theta= \theta + \IN$. Then, the image of the homomorphism
	\begin{equation*}
	\psi(t) = (\overline{\nu_1 t},\ldots,\overline{\nu_N t})
	\end{equation*}
	winds densely around the torus $\TT^N$, whenever the frequency vector $\nu=(\nu_1,\ldots,\nu_N) \in \RN^N$ is nonresonant, i.e.\ rationally independent. It is easy to verify that the dual group of $\TT^N$ is given by
	\begin{equation*}
	(\TT^N)^* = \{(\bar{\theta}_1,\ldots,\bar{\theta}_N) \mapsto e^{2\pi i (k_1\theta_1+\ldots + k_N \theta_N)} : k \in \IN^N \}.
	\end{equation*}
	Therefore, one possible choice for the cross section would be
	\begin{equation*}
	\Sigma = \{ (\bar{\theta}_1,\ldots,\bar{\theta}_N)\in \TT^N : e^{2\pi i \theta_1} = 1  \} = \{0\}\times\TT^{N-1},
	\end{equation*}
	so $\varphi(\bar{\theta}_1,\ldots,\bar{\theta}_N)=e^{2\pi i \theta_1}$.	In this case, consecutive intersections of the flow and $\Sigma$ would be separated by an interval of the length $1/\nu_1$.
		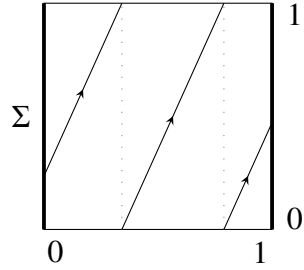
\begin{figure}[h]
		\begin{center}
			
			\begin{tikzpicture}[xscale=3,yscale=3]
			\draw (0,0) rectangle (1,1);
			
			\draw[line width=0.5mm](0,0) -- (0,1);		
			\draw[line width=0.5mm] (1,0) -- (1,1);

			\draw[loosely dotted,gray] (0.34164,0) -- (0.34164,1);		
			\draw[loosely dotted,gray] (0.78885,0) -- (0.78885,1);

			\begin{scope}[>=stealth,decoration={
				markings,
				mark=at position 0.5 with {\arrow{>}}}
			] 
			\draw[postaction={decorate}] (0,0.23606) -- (0.34164,1);
			\draw[postaction={decorate}] (0.34164,0) -- (0.78885,1);
			\draw[postaction={decorate}] (0.78885,0) -- (1,0.47213);
			
			\end{scope}

			\draw (-.1,0.5) node {$\Sigma$};
			
			\draw (0.05,-0.1) node {$0$};
			\draw (0.95,-0.1) node {$1$};
			
			\draw (1.1,0.05) node {$0$};
			\draw (1.1,0.95) node {$1$};

			\end{tikzpicture}
			
		\end{center}
		\caption{On the $2$-torus $\TT^2$, intersections of $\Sigma=\{ 0 \}\times \TT$ and the orbit of $\psi(t)$ are separated by time intervals of length $S=1/\nu_1$.}
		\end{figure}
\end{exam}

\subsection{Almost periodic functions}

The notion of almost periodic functions was introduced by H. Bohr as a generalization of strictly periodic functions \cite{Bohr_1925}. A function $u \in \mathcal{C}(\RN)$ is called \textit{(Bohr) almost periodic}, if for any $\epsilon>0$ there is a relatively dense set of $\epsilon$-almost-periods of this function. By this we mean, that for any $\epsilon>0$ there exists $L=L(\epsilon)$ such that any interval of length $L$ contains at least on number $T$ such that
\begin{equation*}
\lvert u(t+T) - u(t) \rvert < \epsilon \;\; \forall t\in \RN.
\end{equation*}
Later, Bochner \cite{Bochner_1927} gave an alternative but equivalent definition of this property: For a continuous function $u$, denote by $u_\tau(t)$ the translated function $u(t + \tau)$. Then $u$ is (Bohr) almost periodic if and only if every sequence $\left(u_{\tau_n}\right)_{n\in\NN}$ of translations of $u$ has a subsequence that converges uniformly. 

There are several other characterizations of almost periodicity, as well as generalizations due to Stepanov \cite{Stepanoff_1926}, Weyl \cite{Weyl_1927} and Besicovitch \cite{Besicovitch_1926}. In this work we will only consider the notion depicted above and therefore call the corresponding functions just \textit{almost periodic} (\textit{a.p.}). We will however introduce one more way to describe a.p.\ functions using the framework of the previous section: \\
Consider $(\Omega,\psi)$ as above and a function $U \in \mathcal{C}(\Omega)$. Then, the function defined by
\begin{equation} \label{eq representation formula}
u(t)= U(\psi(t))
\end{equation}
is almost periodic. This can be verified easily with the alternative definition due to Bochner. Since $U \in \mathcal{C}(\Omega)$, any sequence $\left(u_{\tau_n}\right)_{n\in\NN}$ will be uniformly bounded and equicontinuous. Hence the Arzelà–Ascoli theorem guarantees the existence of a uniformly convergent subsequence. We will call any function obtainable in this manner \textit{representable over} $(\Omega,\psi)$. Since the image of $\psi$ is assumed to be dense, it is clear that the function $U \in \mathcal{C}(\Omega)$ is uniquely determined by this relation. As an example take $\Omega \cong \SP$, then $\psi$ is periodic. Thus \eqref{eq representation formula} gives rise to periodic functions.
Conversely it is true, that any almost periodic function can be constructed this way. For this purpose we introduce the notion of hull. The hull $\mathcal{H}_u$ of a function $u$ is defined by
\begin{equation*}
\mathcal{H}_u = \overline{\{u_\tau : \tau \in \RN\}},
\end{equation*}
where the closure is taken with respect to uniform convergence on the whole real line. Therefore if $u$ is a.p., then $\mathcal{H}_u$ is a compact metric space. If one uses the continuous extension of the rule
\begin{equation*}
u_\tau * u_s = u_{\tau +s} \;\; \forall \tau,s \in \RN
\end{equation*}
onto all of $\mathcal{H}_u$ as the group operation, then the hull becomes a commutative topological group with neutral element $u$. (For $v,w \in \mathcal{H}_u$ with $v = \lim_{n\to\infty} u_{\tau_n^v}$ and $w = \lim_{n\to\infty} u_{\tau_n^w}$ we have 
\begin{equation*}
v * w = \lim_{n\to\infty} u_{\tau_n^v+\tau_n^v}, \;\; -v = \lim_{n\to\infty} u_{-\tau_n^v}.
\end{equation*} 
These limits exist by Lemma \ref{lem uniform limit} from the appendix. The continuity of both operations can be shown by a similar argument.)
If we further define the flow
\begin{equation*}
\psi_u (\tau) = u_\tau,
\end{equation*}
then the pair $(\mathcal{H}_u,\psi_u)$ matches perfectly the setup of the previous section. Now, the representation formula \eqref{eq representation formula} holds for $U \in \mathcal{C}(\mathcal{H}_u)$ defined by
\begin{equation*}
U(w) = w(0) \;\; \forall w \in \mathcal{H}_u.
\end{equation*}
This function is sometimes called the `extension by continuity' of the almost periodic function $u(t)$ to its hull $\mathcal{H}_u$. This construction is standard in the theory of a.p.\ functions and we refer the reader to \cite{Nemytskii} for a more detailed discussion.\\
\ \\
For a function $U:\Omega \to \RN$ let us introduce the derivative along the flow by 
\begin{equation*}
\p_\psi U(\omega) = \lim_{t \to 0} \frac{U(\omega + \psi(t)) - U(\omega)}{t}.
\end{equation*}
Let $\mathcal{C}^1_\psi(\Omega)$ be the space of continuous functions $U:\Omega \to \RN$ such that $\p_\psi U$ exists for all $\omega \in \Omega$ and $\p_\psi U \in \mathcal{C}(\Omega)$. The spaces $\mathcal{C}^k_\psi(\Omega)$ for $k \geq 2$ are defined accordingly. Let us also introduce the norm $\lVert U \rVert_{\mathcal{C}^k_\psi(\Omega)} = \lVert  U \rVert_\infty + \sum_{n=1}^{k} \lVert \p_\psi^{(n)} U \rVert_\infty$. Now consider $U\in \mathcal{C}(\Omega)$ and assume the almost periodic function $u(t)=U(\psi(t))$ is continuously differentiable. Then $\p_\psi U$ exists on $\psi(\RN)$ and we have
\begin{equation*}
u'(t) = \p_\psi U\left(\psi(t)\right) \;\; \text{ for all } t \in \RN.
\end{equation*}
\begin{lem}
Let $U\in \mathcal{C}(\Omega)$ and $u\in \mathcal{C}(\RN)$ be such that $u(t)=U(\psi(t))$. Then we have $u\in \mathcal{C}^1(\RN)$ and $u'(t)$ is a.p.\ if and only if $U \in \mathcal{C}^1_\psi(\Omega)$.
\end{lem}
One part of the equivalence is trivial. The proof of the other part can be found in \cite[Lemma 13]{Ortega_Tarallo}. We also note that the derivative $u'(t)$ of an almost periodic function is itself a.p.\ if and only if it is uniformly continuous. This, and many other interesting properties of a.p.\ functions are demonstrated in \cite{Besicovitch_1926}.
\begin{exam}
	Let us continue Example \ref{exam torus}, where $\Omega = \TT^N$. For $U \in \mathcal{C}(\TT^N)$ consider the function
	\begin{equation*}
	u(t) = U(\psi(t)) = U(\overline{\nu_1 t},\ldots,\overline{\nu_N t}).
	\end{equation*}
	Such functions are called \textit{quasi-periodic}. In this case, $\p_\psi$ is just the derivative in the direction of $\nu\in\RN^N$. So if $U$ is in the space $\mathcal{C}^1(\TT^N)$ of functions in $\mathcal{C}^1(\RN^N)$, which are $1$-periodic in each argument, then
	\begin{equation*}
	\p_\psi U = \sum_{i=1}^{N} \nu_i \, \p_{\theta_i} U.
	\end{equation*}
	Note however, that in general $\mathcal{C}^1_\psi(\TT^N)$ is a proper subspace of $\mathcal{C}^1(\TT^N)$.
\end{exam}

\subsection{Haar measure and decomposition along the flow} \label{sec decomp haar measure}

It is a well known fact, that for every compact commutative topological group $\Omega$ there is a unique Borel probability measure $\mu_\Omega$, which is invariant under the group operation, i.e.\ $\mu_\Omega(\mathcal{D}+\omega) = \mu_\Omega(\mathcal{D})$ holds for every Borel set $\mathcal{D}\subset \Omega$ and every $\omega \in \Omega$. This measure is called the \textit{Haar measure} of $\Omega$. (This follows from the existence of the invariant \textit{Haar integral} of $\Omega$ and the Riesz representation theorem. Proofs can be found in \cite{Pontryagin_1966} and \cite{Hewitt}, respectively.) For Example if $\Omega=\SP$ we have
\begin{equation*}
\mu_{\SP}(\mathcal{B}) = \frac{1}{2\pi}\lambda\{t\in[0,2\pi):e^{it}\in\mathcal{B}\},
\end{equation*}
where $\lambda$ is the Lebesgue measure on $\RN$. Let $\psi$, $\Sigma$ and $\Phi$ be as in section \ref{sec compact top groups}. Then $\Phi$ defines a decomposition $\Omega \cong \Sigma \times [0,S)  $ along the flow. Since $\Sigma$ is a subgroup, it has a Haar measure $\mu_\Sigma$ itself. Also the interval $[0,S)$ naturally inherits the probability measure 
\begin{equation*}
\mu_{[0,S)}(I) = \frac{1}{S} \lambda(I).
\end{equation*}
As shown in \cite{Campos_2013}, the restricted flow $\Phi:\Sigma\times[0,S)\to\Omega, \Phi(\sigma,t)= \sigma \cdot t$ also allows for a decomposition of the Haar measure $\mu_\Omega$ along the flow. 
\begin{lem} \label{lem decomp of haar measure}
	The map $\Phi$ is an isomorphism of measure spaces, i.e.
	\begin{equation} \label{eq decomposition haar measure}
	\mu_\Omega(\mathcal{B}) = \frac{1}{S} (\mu_\Sigma \otimes \lambda) (\Phi^{-1}(\mathcal{B}))
	\end{equation}
	holds for every Borel set $\mathcal{B} \subset \Omega$.
\end{lem}
Before we prove this lemma, let us begin with some preliminaries. Consider the function $\chi: \Sigma\times [0,\infty)  \to \Sigma\times [0,S)$ defined by
\begin{equation} \label{eq def chi}
\chi(\sigma,t) = \Phi^{-1}(\sigma \cdot t) = \Phi^{-1}(\sigma +\psi(t)).
\end{equation}
Since $\Phi$ is just the restricted flow, we have $\chi=\id$ on $\Sigma\times[0,S)$. This yields
\begin{equation*}
\chi(\sigma,t) = \Phi^{-1}(\sigma+ \psi(t)) 
= \Phi^{-1}\left(\sigma+ \psi\left(\left\lfloor \frac{t}{S}\right\rfloor S \right)  + \psi\left(t-\left\lfloor\frac{s}{S}\right\rfloor S\right)  \right)
= \left(\sigma + \psi\left(\left\lfloor \frac{t}{S}\right\rfloor S \right), t-\left\lfloor\frac{t}{S}\right\rfloor S \right)
\end{equation*}
for every $(\sigma,t)\in\Sigma\times\RN$, where $\lfloor\cdot\rfloor$ indicates the floor function. This representation shows that $\chi$ is measure-preserving on every strip $ \Sigma\times [t,t+S)$ of width $S$, since $\mu_{\Sigma}$ and $\lambda$ are invariant under translations in $\Sigma$ and $\RN$, respectively.
Moreover, the equality
\begin{equation} \label{eq ident chi phi}
\chi(\Phi^{-1}(\omega) + \Phi^{-1}(\tilde\omega)) = \Phi^{-1}(\omega + \tilde{\omega}) \;\; \forall \omega,\tilde{\omega}\in \Omega
\end{equation}
follows directly from the definition of $\chi$.
	
	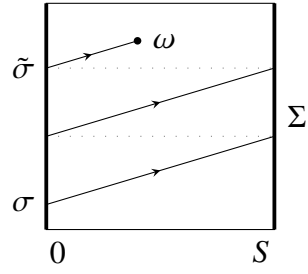
\begin{figure}[h]
\begin{center}
	
	\begin{tikzpicture}[xscale=3,yscale=3]

	\draw[loosely dotted,gray] (0,0.41209) -- (1,0.41209);
	\draw[loosely dotted,gray] (0,0.7136) -- (1,0.7136);
	
	\begin{scope}[acteur/.style={circle, fill=black,thin, inner sep=1pt, minimum size=0.1cm}]
	\node (a1) at (0.4,0.83421) [acteur,label=right:$\omega$]{};
	\end{scope}
	
	\node[acteur/.style={circle, fill=black,thin, inner sep=0pt, minimum size=0.01cm}] (a2) at (0,0.11058) [acteur,label=left:$\sigma$]{};
	\node[acteur/.style={circle, fill=black,thin, inner sep=0pt, minimum size=0.01cm}] (a2) at (0,0.7136) [acteur,label=left:$\tilde\sigma$]{};

	\draw (0.05,-0.1) node {$0$};
	\draw (0.95,-0.1) node {$S$};
	
	
	\draw (1.1,0.5) node {$\Sigma$};

	\begin{scope}[>=stealth,decoration={
		markings,
		mark=at position 0.5 with {\arrow{>}}}
	] 
	\draw[postaction={decorate}] (0,0.11058) -- (1,0.41209);
	\draw[postaction={decorate}] (0,0.41209) -- (1,0.7136);
	\end{scope}
	
	\begin{scope}[>=stealth,decoration={
		markings,
		mark=at position 0.5 with {\arrow{>}}}
	] 
	\draw[postaction={decorate}] (0,0.7136) -- (0.4,0.83421);
	\end{scope}

	\draw (0,0) rectangle (1,1);
	
	\draw[line width=0.5mm](0,0) -- (0,1);		
	\draw[line width=0.5mm] (1,0) -- (1,1);
	

	

	
	%
	%
	%
	%
	%
	
	\end{tikzpicture}

\end{center}		
	\caption{Let $\chi(\sigma,t)=(\tilde{\sigma},s)$. The map $\chi$ `divides out' every complete period of $\varphi\circ \psi$, i.e.\ $s = t\mod S$, while preserving the relation $\tilde{\sigma}\cdot s = \omega = \sigma\cdot t$.}

	\end{figure}

\begin{proof}[Proof of Lemma \ref{lem decomp of haar measure}]
	First we show that $\Phi^{-1}$ is Borel measurable. To prove this, it suffices to show that the image $\Phi(A\times I)$ of every open rectangle $A\times I\subset \Sigma \times [0,S)$ is a Borel set. If $0 \notin I$ this image is open in $\Omega\setminus\Sigma$, since $\Phi^{-1}$ is continuous. But if $0 \in I$, again $\Phi(A\times (I\setminus\{0\}))$ is open and $\Phi(A\times \{0\}) = A$ is it as well.\\
	Now, consider the measure $\mu_\Phi$ on $\Omega$ defined by
		\begin{equation} 
		\mu_\Phi(\mathcal{B}) = \frac{1}{S} (\mu_\Sigma \otimes \lambda) (\Phi^{-1}(\mathcal{B})).
		\end{equation}
	Since $\mu_\Phi(\Omega)=1$, this is a Borel probability measure. We will show that $\mu_\Phi$ is also invariant under addition in the group. For this purpose, let $\mathcal{B} \subset \Omega$ be a Borel set and let $\omega_0 \in \Omega$. Then, by \eqref{eq ident chi phi} we have
		\begin{align}
		\mu_\Phi(\mathcal{B}+\omega_0) 
		= \frac{1}{S} (\mu_\Sigma \otimes \lambda) (\Phi^{-1}(\mathcal{B}+\omega_0))\\
		= \frac{1}{S} (\mu_\Sigma \otimes \lambda) \left(\chi(\Phi^{-1}(\mathcal{B}) + \Phi^{-1}(\omega_0))\right).
		\end{align}	
	 Denoting $\Phi^{-1}(\omega_0)=(\sigma_0,s_0)$, we get $\Phi^{-1}(\mathcal{B}) + \Phi^{-1}(\omega_0)\subset \Sigma\times[s_0,s_0+S)$. So it is contained in a strip of width $S$ and therefore
	 \begin{align*}
		\frac{1}{S} (\mu_\Sigma \otimes \lambda) \left(\chi(\Phi^{-1}(\mathcal{B}) + (\sigma_0,s_0))\right)
		= \frac{1}{S} (\mu_\Sigma \otimes \lambda) \left(\Phi^{-1}(\mathcal{B}) + (\sigma_0,s_0)\right)
	 \end{align*}
	But the product measure $\mu_\Sigma \otimes \lambda$ is invariant under translations in $\Sigma\times \RN$. Thus, in total we have
	\begin{align}
	\mu_\Phi(\mathcal{B}+\omega_0) 
	= \frac{1}{S} (\mu_\Sigma \otimes \lambda) \left(\Phi^{-1}(\mathcal{B})\right)
	=\mu_\Phi(\mathcal{B}) .
	\end{align}
	Therefore, $\mu_\Phi$ is a Borel probability measure on $\Omega$ which is invariant under group action. Since the Haar measure is unique, it follows $\mu_\Omega = \mu_\Phi$.
\end{proof}

\section{A theorem about escaping sets} \label{sec a theorem about escaping sets}
\subsection{Measure-preserving embeddings}

From now on we will consider functions
\begin{equation*}
f:\mathcal{D}\subset \Omega\times(0,\infty) \to \Omega\times(0,\infty),
\end{equation*}
where $\mathcal{D}$ is an open set. We will call such a function \textit{measure-preserving embedding}, if $f$ is continuous, injective and furthermore
\begin{equation*}
(\mu_{\Omega} \otimes \lambda) (f(\mathcal{B}))= (\mu_{\Omega} \otimes \lambda) (\mathcal{B})
\end{equation*}
holds for all Borel sets $\mathcal{B}\subset \mathcal{D}$, where $\lambda$ denotes the Lebesgue measure of $\RN$. It is easy to show that under these conditions, $f:\mathcal{D} \to \tilde{\mathcal{D}}$ is a homeomorphism, where $\tilde{\mathcal{D}}=f(\mathcal{D})$. \\
Since we want to use the iterations of $f$, we have to carefully construct a suitable domain on which these forward iterations are well-defined. We initialize $\mathcal{D}_1 = \mathcal{D}, \;\; f^1=f$ and set
\begin{equation*}
\mathcal{D}_{n+1}=f^{-1}(\mathcal{D}_{n}), \;\; f^{n+1}= f^n \circ f \; \text{ for  }  \; n \in \NN.
\end{equation*}
This way $f^n$ is well-defined on $\mathcal{D}_n$. Clearly, $f^n$ is a measure-preserving embedding as well. Also inductively it can be shown that $\mathcal{D}_{n+1}=\{(\omega,r)\in\mathcal{D}:f(\omega,r),\ldots,f^n(\omega,r)\in \mathcal{D} \}$
and therefore $\mathcal{D}_{n+1}\subset \mathcal{D}_n \subset \mathcal{D}$ for all $n \in \NN$.
Initial conditions in the set
\begin{equation*}
\mathcal{D}_\infty = \bigcap\limits_{n=1}^\infty \mathcal{D}_n \subset \Omega\times(0,\infty)
\end{equation*}
correspond to complete forward orbits, i.e.\ if $(\omega_0,r_0)\in \mathcal{D}_\infty$, then
\begin{equation*}
(\omega_n,r_n)=f^n(\omega_0,r_0)
\end{equation*}
is defined for all $n \in \NN$. It could however happen that $\mathcal{D}_\infty = \emptyset$ or even $\mathcal{D}_{n} = \emptyset$ for some $n\geq 2$. The set of initial data leading to unbounded orbits is denoted by
\begin{equation} \label{eq U unbounded orbits}
\mathcal{U} = \{ (\omega_0,r_0)\in \mathcal{D}_\infty : \limsup_{n\to\infty} r_n = \infty \}.
\end{equation}
Complete orbits such that $\lim_{n\to\infty} r_n = \infty$ will be called \textit{escaping orbits}. The corresponding set of initial data is
\begin{equation*}
\mathcal{E} = \{ (\omega_0,r_0)\in \mathcal{D}_\infty : \lim_{n\to\infty} r_n = \infty \}.
\end{equation*}

\subsection{Almost periodic successor maps} \label{sec a p successor maps}

Now, consider a measure-preserving embedding $f:\mathcal{D}\subset \Omega \times(0,\infty) \to \Omega\times(0,\infty)$, which has the special structure
\begin{equation} \label{embedding form}
f(\omega,r)=(\omega+\psi(F(\omega,r)),r + G(\omega,r)),
\end{equation}
where $F,G:\mathcal{D}\to \RN$ are continuous. For $\omega \in \Omega$ we introduce the notation $\psi_\omega(t) = \omega + \psi(t) = \omega \cdot t$ and define
\begin{equation*}
{D}_{\omega} = (\psi_{\omega} \times \id)^{-1}(\mathcal{D}) \subset \RN \times (0,\infty).
\end{equation*}
On this open set, consider the map $f_{\omega}: {D}_{\omega} \subset \RN \times (0,\infty) \to  \RN \times (0,\infty)$ given by
\begin{equation} \label{planar maps}
f_{\omega}(t,r)=(t+F(\psi_\omega(t),r),r+G(\psi_\omega(t),r)).
\end{equation}
Then $f_{\omega}$ is continuous and meets the identity
\begin{equation*}
f \circ (\psi_{\omega} \times \id) = (\psi_{\omega} \times \id) \circ f_{\omega} \;\; \text{on} \;\; D_{\omega},
\end{equation*} 
i.e.\ the following diagram is commutative:
\begin{equation*}
\begin{tikzcd}
\mathcal{D} \arrow[r, "f"] 
& f(\mathcal{D}) \subset \TT^N\times(0,\infty)  \\
{D}_{\omega} \arrow[r, "f_{\omega}"] \arrow[u, "\psi_{\omega} \times \id"]
& f_{\omega}( {D}_{\omega}) \subset \RN \times (0,\infty) \arrow[u, "\psi_{\omega} \times \id"]
\end{tikzcd}
\end{equation*}
Therefore $f_{\omega}$ is injective as well. Again we define $D_{\omega,1} = D_{\omega}$ and $D_{\omega,n+1} = f_{\omega}^{-1}(D_{\omega,n})$ to construct the set
\begin{equation*}
D_{\omega, \infty} = \bigcap\limits_{n=1}^{\infty} D_{\omega,n} \subset \RN \times (0,\infty),
\end{equation*}
where the forward iterates $(t_n,r_n)=f_{\omega}^n(t_0,t_0)$ are defined for all $n\in\NN$. Analogously, unbounded orbits are generated by initial conditions in the set
\begin{equation*}
U_\omega = \{(t_0,r_0)\in{D}_{{\omega},\infty}: \limsup_{n\to\infty}r_n=\infty  \}
\end{equation*}
and escaping orbits originate in 
\begin{equation*}
{E_{\omega}}= \{(t_0,r_0)\in{D}_{{\omega},\infty}: \lim_{n\to\infty}r_n=\infty  \}.
\end{equation*}
These sets can also be obtained through the relations
\begin{align*}
D_{\omega, \infty} =  (\psi_{\omega} \times \id)^{-1}(\mathcal{D}_\infty), \;\; U_{\omega} =  (\psi_{\omega} \times \id)^{-1}(\mathcal{U}), \;\; E_{\omega} =  (\psi_{\omega} \times \id)^{-1}(\mathcal{E}).
\end{align*}
Finally we are in position to state the theorem \cite[Theorem~3.1]{kunze_ortega_ping_pong}:
\begin{thm} \label{thm escaping set nullmenge}
	Let $f:\mathcal{D}\subset \Omega\times (0,\infty) \to \Omega \times (0,\infty)$ be a measure-preserving embedding of the form (\ref{embedding form}) and suppose that there is a function $W=W(\omega,r)$ satisfying $W\in \mathcal{C}^1_\psi(\Omega \times (0,\infty))$,
	\begin{equation} \label{W growth}
	0<\beta \leq \p_r W(\omega,r) \leq \delta \;\; \text{for} \;\; \omega \in \Omega, \;\; r \in (0,\infty),
	\end{equation}
	with some constants $\beta,\delta>0$, and furthermore
	\begin{equation} \label{ineq adiabatic inv}
	W(f(\omega,r)) \leq W(\omega,r) + k(r) \;\; \text{for} \;\; (\omega,r)\in \mathcal{D},
	\end{equation}
	where $k:(0,\infty) \to \RN$ is a decreasing and bounded function such that $\lim_{r\to\infty} k(r)=0$. Then, for almost all $\omega \in \Omega$, the set $E_{\omega}\subset \RN \times (0,\infty)$ has Lebesgue measure zero.
\end{thm}
Here, $\mathcal{C}^1_\psi(\Omega \times (0,\infty))$ denotes the space of functions $U(\omega,r)$ such that $U(\cdot,r)\in \mathcal{C}^1_\psi(\Omega)$ and $U(\omega,\cdot)\in \mathcal{C}^1(0,\infty)$ for every $(\omega,r)\in \Omega\times \RN$.
The function $W$ can be seen as a generalized adiabatic invariant, since any growth will be slow for large energies.

\section{Proof of Theorem \ref{thm escaping set nullmenge}}

The proof of Theorem \ref{thm escaping set nullmenge} is based on the fact, that almost all unbounded orbits of $f$ are recurrent. In order to show this, we will apply the Poincaré recurrence theorem to the set $\mathcal{U}$ of unbounded orbits and the corresponding restricted map $f\big|_\mathcal{U}$. We will use it in the following form \cite[Lemma 4.2]{kunze_ortega_ping_pong}.
\begin{lem} \label{lem poincare recurrence theorem}
	Let $(X,\mathcal{F},\mu)$ be a measure space such that $\mu(X)<\infty$. Suppose that there exists a measurable set $\Gamma\subset X$ of measure zero and a map $T:X\setminus \Gamma\to X$ which is injective and so that the following holds:
	\begin{enumerate} [label=(\alph*)]
		\item $T$ is measurable, in the sense $T(B),T^{-1}(B) \in \mathcal{F}$ for $B\in\mathcal{F}$, and
		\item $T$ is measure-preserving, in the sense that $\mu(T(B))=\mu(B)$ for $B\in\mathcal{F}$. 
	\end{enumerate}
	Then for every measurable set $B\subset X$ almost all points of $B$ visit $B$ infinitely many times in the future (i.e.\ $T$ is infinitely recurrent).
\end{lem}
Since we can not guarantee that $\mathcal{U}$ has finite measure, we will also need the following refined version of the recurrence theorem due to Dolgopyat \cite[Lemma 4.3]{Dolgopyat_lectures}.
\begin{lem} \label{lem dolgopyat}
	Let $(X,\mathcal{F},\mu)$ be a measure space and suppose that the map $T:X\to X$ is injective and such that the following holds:
	\begin{enumerate} [label=(\alph*)]
		\item $T$ is measurable, in the sense $T(B),T^{-1}(B) \in \mathcal{F}$ for $B\in\mathcal{F}$,
		\item $T$ is measure-preserving, in the sense that $\mu(T(B))=\mu(B)$ for $B\in\mathcal{F}$, and
		\item there is a set $A\in\mathcal{F}$ such that $\mu(A)<\infty$ with the property that almost all points from $X$ visit $A$ in the future.
	\end{enumerate}
	Then for every measurable set $B\subset X$ almost all points of $B$ visit $B$ infinitely many times in the future (i.e.\ $T$ is infinitely recurrent).
\end{lem}
For the sake of completeness let us state the proof.
\begin{proof}[Proof of Lemma \ref{lem dolgopyat}]
	Let $\Gamma \subset X$ be measurable such that $\mu(\Gamma)=0$ and all points of $X\setminus \Gamma$ vist $A$ in the future. Thus, the first return time $r(x)=\min \{ k\in\NN : T^k(x)\in A \}$ is well-defined for $x \in X \setminus \Gamma$. It induces a map $S:X \setminus \Gamma \to A$ defined by $S(x) = T^{r(x)}(x)$. The restriction $S\big|_{A\setminus \Gamma}$ is injective: Assume $S(x)=S(y)$ for distinct points $x,y\in A\setminus \Gamma$ and suppose $r(x)>r(y)$, then $T^{r(x)-r(y)}(x) = y \in A$ is a contradiction to the minimality of $r(x)$. It is also measure-preserving \cite[cf. Lemma 2.43]{Einsiedler2011}. Now, consider a measurable set $B\subset X$ and define $B_j = \{ y \in B\setminus \Gamma : r(y) \leq j\}$ as well as
	\begin{equation*}
	A_j = S(B_j) = \bigcup_{k=1}^j (T^k(B)\cap A) \subset A \;\; \forall j \in \NN.
	\end{equation*}
	But since $\mu(A)< \infty$ by assumption, the Poincaré recurrence theorem (Lemma \ref{lem poincare recurrence theorem}) applies to $A_j$. Thus we can find measurable sets $\Gamma_j \subset A_j$ with measure zero, such that every point $x \in A_j\setminus \Gamma_j$ returns to $A_j$ infinitely often (via $S$). Now consider the set 
	\begin{equation*}
	F = B \cap \bigg( \Gamma \cup \bigcup_{j\in\IN} S^{-1}(\Gamma_j)  \bigg).
	\end{equation*}
	Then $\mu(F)=0$ and every point $y \in B\setminus F$ returns to $B$ infinitely often in the future. To see this, select $j \in \NN$ such that $r(y)\leq j$, i.e.\ $y \in B_j$. Then $x = S(y) \in A_j \setminus \Gamma_j$. Hence there exist infinitely many $k \in \NN$ so that $k \geq j$ and $S^k(x) \in A_j$. Let us fix one of these $k$. Then $S^k(x)=S(z)$ for some $z \in B_j$. So in total we have
	\begin{equation*}
	T^{r(z)}(z) = S(z) = S^k(x) = S^{k+1}(y) = T^{\sum_{j=1}^{k} r(S^j(y))}(y).
	\end{equation*}
	Now, since $\sum_{j=1}^{k} r(S^j(y)) \geq k+1 > j \geq r(z)$, this yields $T^m(y) = z \in B_j \subset B$, where $m = \sum_{j=1}^{k} r(S^j(y)) -r(z) \in \NN$.
\end{proof}

One way to construct such a set $A$ of finite measure is given by the next lemma \cite{kunze_ortega_ping_pong}. It is based on the function $W(\omega,r)$ introduced in Theorem \ref{thm escaping set nullmenge} and in fact is the only reason to assume the existence of $W$ in the first place.
\begin{lem} \label{lem finite measure set}
	Let $f:\mathcal{D}\subset \Omega \times (0,\infty) \to \Omega \times (0,\infty)$ be a measure-preserving embedding and suppose that there is a function $W=W(\omega,r)$ satisfying $W \in \mathcal{C}^1_\psi(\Omega \times (0,\infty))$, (\ref{W growth}) and (\ref{ineq adiabatic inv}). Let $(\epsilon_j)_{j\in\NN}$ and $(W_j)_{j\in\NN}$ be sequences of positive numbers with the properties $\sum_{j=1}^{\infty}\epsilon_j < \infty$, $\lim_{j\to\infty} W_j = \infty$ and $\lim_{j\to\infty} \epsilon_j^{-1} k(\frac{1}{4\gamma}W_j) = 0$. Denote
	\begin{equation}
	\mathcal{A} = \bigcup_{j\in\NN} \mathcal{A}_j, \;\; \mathcal{A}_j = \{(\omega,r)\in \Omega \times (0,\infty): \lvert W(\omega,r) - W_j \rvert \leq \epsilon_j \}.
	\end{equation}
	Then $\mathcal{A}$ has finite measure and every unbounded orbit of $f$ enters $\mathcal{A}$. More precisely, if $(\omega_0,r_0)\in \mathcal{U}$, where $\mathcal{U}$ is from \eqref{eq U unbounded orbits}, and if $(\omega_n,r_n)_{n\in\NN}$ denotes the forward orbit under $f$, then there is $K\in\NN$ so that $(\omega_K,r_K) \in \mathcal{A}$.
\end{lem}
\begin{proof}
	First let us show that $\mathcal{A}$ has finite measure. By Fubini's theorem,
	\begin{equation*}
	(\mu_{\Omega}\otimes \lambda)(\mathcal{A}_j) = \int_\Omega \lambda(\mathcal{A}_{j,\omega}) \, d\mu_\Omega(\omega)
	\end{equation*}
	holds for the sections $\mathcal{A}_{j,\omega} = \{ r \in(0,\infty): (\omega,r) \in \mathcal{A}_j \}$. Now, consider the diffeomorphism $w_\omega:r \mapsto W(\omega,r)$. Its inverse $w_\omega^{-1}$ is Lipschitz continuous with constant $\beta^{-1}$, due to \eqref{W growth}. But then, $\mathcal{A}_{j,\omega} = w_\omega^{-1}( (W_j-\epsilon_j, W_j + \epsilon_j) )$ implies $\lambda(\mathcal{A}_{j,\omega})\geq 2\beta^{-1}\epsilon_j$. Thus in total we have
	\begin{equation*}
	(\mu_\Omega \otimes \lambda)(\mathcal{A}) \leq \sum_{j=1}^{\infty} (\mu_\Omega \otimes \lambda)(\mathcal{A}_j) \leq \sum_{j=1}^{\infty} \frac{2\epsilon_j}{\beta} < \infty.
	\end{equation*}
	Next we will prove the recurrence property. To this end, let $(\omega_0,r_0) \in \mathcal{U}$ be fixed and denote by $(\omega_n,r_n)$ the forward orbit under $f$. We will start with some preliminaries. Using \eqref{W growth} and the mean value theorem, we can find $\hat{r}$ such that
	\begin{equation} \label{ineq W durch r}
	\frac{\beta}{2} \leq \frac{W(\omega,r)}{r} \leq 2\delta \;\; \forall (\omega,r)\in \Omega \times (\hat{r},\infty).
	\end{equation}
	Furthermore, by assumption we can find an index $j_0\geq 2$ such that
	\begin{equation*}
	W_{j_0} > \max \{ W(\omega_1,r_1),  \lVert k \rVert_\infty + \max_{\omega \in \Omega} W(\omega,\hat{r}) , 2 \lVert k \rVert_\infty\} \;\; \text{and} \;\; k\bigg(\frac{1}{4\gamma}W_{j_0}\bigg) \leq \epsilon_{j_0}.
	\end{equation*}
	Moreover we have $\limsup_{n \to \infty} W(\omega_n,r_n) = \infty$: Due to $\limsup_{n \to \infty} r_n = \infty$, \eqref{W growth} implies
	\begin{equation*}
	W(\omega_n,r_n) \geq \beta (r_n-r_1) + W(\omega_n,r_1) 
	\end{equation*}
 	for $n$ sufficiently large. But then $\limsup_{n \to \infty} W(\omega_n,r_n) = \infty$ follows from the compactness of $\Omega$. Now, since $W(\omega_1,r_1)<W_{j_0}$ we can select the first index $K\geq 2$ such that $W(\omega_K,r_K) > W_{j_0}$. So in particular this means $W(\omega_{K-1},r_{K-1}) \leq W_{j_0}$. Since \eqref{ineq adiabatic inv} yields $W(\omega_K,r_K) \leq W(\omega_{K-1},r_{K-1}) + k(r_{K-1})$, we can derive the following inequality:
 	\begin{equation*}
	W(\omega_{K-1},r_{K-1}) \geq W(\omega_{K},r_{K}) -  \lVert k \rVert_\infty > W_{j_0} -  \lVert k \rVert_\infty \geq \max_{\omega \in \Omega} W(\omega,\hat{r}) \geq W(\omega_{K-1},\hat{r})
 	\end{equation*}
 	Then, the monotonicity of $w_{\omega_{K-1}}$ implies $r_{K-1} > \hat{r}$. Hence we can combine \eqref{ineq W durch r} with the previous estimate to obtain
 	\begin{equation*}
 	r_{K-1}\geq \frac{1}{2\delta} W(\omega_{K-1},r_{K-1}) \geq  \frac{1}{2\delta} (W_{j_0} - \lVert k \rVert_\infty) \geq \frac{1}{4\delta} W_{j_0}.
 	\end{equation*}
 	Finally, since $k(r)$ is decreasing, $W(\omega_{K},r_{K}) > W_{j_0} \geq W(\omega_{K-1},r_{K-1})$ yields
 	\begin{equation*}
 	\lvert 	W(\omega_{K},r_{K}) - W_{j_0} \rvert \leq 	W(\omega_{K},r_{K}) - 	W(\omega_{K-1},r_{K-1})  \leq k(r_{K-1}) \leq k\bigg(\frac{1}{4\delta} W_{j_0}\bigg) \leq \epsilon_{j_0},
 	\end{equation*}
 	which implies $(\omega_K,r_K) \in \mathcal{A}_{j_0}$.	
\end{proof}
Now, we are ready to prove the theorem.
\begin{proof}[Proof of Theorem \ref{thm escaping set nullmenge}]
	Consider the set
		\begin{equation*}
		\mathcal{U} = \{ (\omega_0,r_0) \in \mathcal{D}_\infty : \limsup_{n \to \infty} r_n = \infty \}.
		\end{equation*}
	We will assume that $\mathcal{U}\neq \emptyset$, since otherwise the assertion would be a direct consequence. 
	\underline{Step 1}: Almost all unbounded orbits are recurrent. We will prove the existence of a set $\mathcal{Z}\subset \mathcal{U}$ of measure zero such that if $(\omega_0,r_0)\in \mathcal{U}\setminus \mathcal{Z}$, then
		\begin{equation*}
		\liminf_{n\to\infty} r_n <\infty.
		\end{equation*}
	In particular, we would have $\mathcal{E} \subset \mathcal{Z}$. To show this, we consider the restriction $T=f\big|_\mathcal{U}:\mathcal{U}\to\mathcal{U}$. This map is well-defined, injective and, like $f$, measure-preserving. We will distinguish three cases:
		\begin{enumerate} [label=(\roman*)]
			\item $(\mu_\Omega \otimes \lambda)(\mathcal{U})=0$,
			\item $0<(\mu_\Omega \otimes \lambda)(\mathcal{U})<\infty$, and
			\item $(\mu_\Omega \otimes \lambda)(\mathcal{U})=\infty$.
		\end{enumerate}
	In the first case $\mathcal{Z}=\mathcal{U}$ is a valid choice. In case $(ii)$ we can apply the Poincaré recurrence theorem (Lemma \ref{lem poincare recurrence theorem}), whereas in case $(iii)$ the modified version of Dolgopyat (Lemma \ref{lem dolgopyat}) is applicable due to Lemma \ref{lem finite measure set}. Now, let us cover $\Omega\times \RN$ by the sets $\mathfrak{B}_j = \Omega \times (j-1,j+1)$ for $j \in \NN$. Then, for $\mathcal{B}_j = \mathfrak{B}_j \cap \mathcal{U}$ one can use the recurrence property to find sets $\mathcal{Z}_j\subset\mathcal{B}_j$ of measure zero such that every orbit $(\omega_n,r_n)_{n\in\NN}$ starting in $\mathcal{B}_j \setminus \mathcal{Z}_j$ returns to $\mathcal{B}_j$ infinitely often. But this implies $\liminf_{n\to\infty} r_n \leq r_0 +2<\infty$. Therefore, the set $\mathcal{Z}= \bigcup_{j\in\NN} \mathcal{Z}_j \subset \mathcal{U}$ has all the desired properties. \\
	\underline{Step 2}: The assertion is valid on the subgroup $\Sigma\subset \Omega$. 
	Since $\mathcal{E}\subset\mathcal{Z}$ by construction, the inclusion
		\begin{equation*}
		 E_\omega = (\psi_\omega \otimes \id)^{-1}(\mathcal{E}) \subset (\psi_\omega \otimes \id)^{-1}(\mathcal{Z})
		\end{equation*}
	holds for all $\omega \in \Omega$. To $j\in\IN$ we can consider the restricted flow
		\begin{equation*}
		\Phi_j : \Sigma \times [jS,(j+1)S) \to \Omega, \;\; \Phi_j(\sigma,t) = \sigma \cdot t = \psi_\sigma(t).
		\end{equation*}
	It is easy to verify that just like $\Phi=\Phi_0$ of Lemma \ref{lem decomp of haar measure} those functions are isomorphisms of measure spaces. In other words, $\Phi_j$ is bijective up to a set of measure zero, both $\Phi_j$ and $\Phi_j^{-1}$ are measurable, and for every Borel set $\mathcal{B} \subset \Omega$ we have
		\begin{equation} \label{eq measure preserving phi}
		\mu_\Omega(\mathcal{B}) = \frac{1}{S} (\mu_\Sigma \otimes \lambda) (\Phi^{-1}_j(\mathcal{B})).
		\end{equation}
	This clearly implies
		\begin{equation} \label{eq measure preserving phi times id}
		(\mu_\Omega \otimes \lambda)(B) = \frac{1}{S} (\mu_\Sigma \otimes \lambda^2)  (\Phi^{-1}_j\times \id) (B)
		\end{equation}
	for every Borel set $B \subset \Omega\times (0,\infty)$. Let
		\begin{equation*}
		C_j = \{ (\sigma,t,r) \in \Sigma\times [jS,(j+1)S) \times (0,\infty) : (\Phi_j(\sigma,t),r) \in \mathcal{Z} \}  = (\Phi_j^{-1}\times \id)(\mathcal{Z}).
		\end{equation*}
	Since $\mathcal{Z}$ has measure zero, \eqref{eq measure preserving phi times id} yields $(\mu_\Sigma\otimes \lambda^2) (C_j) = 0$. Next we consider the cross sections
		\begin{equation*}
		C_{j,\sigma} = \{ (t,r) \in  [jS,(j+1)S) \times (0,\infty) : (\sigma,t,r) \in C_j \}.
		\end{equation*}	
	Then, $\lambda^2(C_{j,\sigma})=0$ for $\mu_\Sigma$-almost all $\sigma \in \Sigma$ follows from Fubini's theorem. So for every $j\in\IN$ there is a set $M_j \subset \Sigma$ with $\mu_\Sigma(M_j)=0$ such that $\lambda^2(C_{j,\sigma}) = 0$ for all $\sigma \in \Sigma \setminus M_j$. Thus $M= \bigcup_{j\in\IN}M_j$ has measure zero as well and
		\begin{equation*}
		\lambda^2\bigg(\bigcup_{j\in\IN} C_{j,\sigma} \bigg) = 0
		\end{equation*}
	for all $\sigma \in \Sigma\setminus M$. But we have
		\begin{equation*}
		\bigcup_{j\in\IN} C_{j,\sigma} = \{ (t,r) \in \RN\times(0,\infty) : (\psi_\sigma (t),r)\in \mathcal{Z} \} = (\psi_\sigma \times \id)^{-1}(\mathcal{Z}),
		\end{equation*}
	and recalling that $E_\sigma \subset (\psi_\sigma \times \id)^{-1}(\mathcal{Z})$, we therefore conclude $\lambda^2(E_\sigma)=0$ for all $\sigma \in \Sigma \setminus M$. \\
	\underline{Step 3}: Concluding from $\Sigma$ to $\Omega$. If we denote by $T_s(t,r)=(t+s,r)$ the  translation in time, then clearly
		\begin{equation*}
		f_{\omega \cdot s} = T_{-s}\circ f_\omega \circ T_s \;\; \text{on} \;\; D_{\omega \cdot s}
		\end{equation*}
	holds for all $\omega \in \Omega$ and $s\in\RN$. But this implies $T_s(E_{\omega \cdot s}) = E_\omega$, since the identity above stays valid under iterations. In particular we have
		\begin{equation*}
		\lambda^2(E_{\omega \cdot s}) = \lambda^2(E_\omega), \;\;\forall\omega\in\Omega,s\in\RN.
		\end{equation*}
	Again, we consider the restricted flow $\Phi:\Sigma\times [0,S) \to \Omega$, $\Phi(\omega,t)=\omega \cdot t$. Using $M\subset \Sigma$ of Step 2 we define $Z_* = \Phi(M\times[0,S)) \subset \Omega$. Then, \eqref{eq measure preserving phi} and $\mu_\Sigma(M)=0$ imply that also $Z_*$ has measure zero. Now let $\omega \in \Omega \setminus Z_*$ be fixed and let $(\sigma,\tau) = \Phi^{-1}(\omega)$. Then $\sigma \in \Sigma\setminus M$ and $\sigma \cdot \tau = \omega$. Therefore, Step 2 implies
		\begin{equation*}
		\lambda^2(E_\omega)= \lambda^2(E_{\sigma \cdot \tau}) = \lambda^2(E_\sigma) = 0,
		\end{equation*}
	which proves the assertion.
\end{proof}

\section{Statement and proof of the main result}

We start with a rigorous description of the ping-pong map. To this end, let $p$ be a forcing such that
\begin{equation} \label{general forcing function}
p \in \mathcal{C}^2(\RN), \;\; 0<a\leq p(t) \leq b \;\;\forall t \in \RN, \;\; \lVert p\rVert_{\mathcal{C}^2} = \lVert p \rVert_\infty + \lVert \dot p \rVert_\infty + \lVert \ddot p \rVert_\infty < \infty.
\end{equation}
Now, we consider the map
\begin{equation*}
(t_0,v_0) \mapsto (t_1,v_1),
\end{equation*}
which sends a time $t_0$ of impact to the left plate $x=0$ and the corresponding velocity $v_0>0$ immediately after the impact to their successors $t_1$ and $v_1$ describing the subsequent impact to $x=0$. If we further denote by $\tilde{t}\in (t_0,t_1)$ the time of the particle's impact to the moving plate, then we can determine $\tilde{t}=\tilde{t}(t_0,v_0)$ implicitly through the equation
\begin{equation} \label{eq def t tilde v}
(\tilde{t}-t_0)v_0 = p(\tilde{t}),
\end{equation}
since this relation describes the distance that the particle has to travel before hitting the moving plate. With that we derive a formula for the successor map:
\begin{equation} \label{eq ping pong map t v}
t_1 = \tilde{t} + \frac{p(\tilde{t})}{v_1}, \;\; v_1 = v_0-2\dot{p}(\tilde{t})
\end{equation}
To ensure that this map is well defined, we will assume that
\begin{equation} \label{eq def v stern}
v_0 > v_* := 2 \max \{ \sup_{t\in\RN} \dot{p}(t),0 \}.
\end{equation}
This condition guarantees that $v_1$ is positive and also implies that there is a unique solution $\tilde{t}=\tilde{t}
(t_0,v_0) \in \mathcal{C}^1(\RN \times (v_*,\infty))$ to \eqref{eq def t tilde v}. Thus we can take $\RN\times (v_*,\infty)$ as the domain of the ping-pong map \eqref{eq ping pong map t v}. Now, we are finally ready to state the main theorem.
\begin{thm} \label{thm main theorem almost periodic}
	Assume $0<a<b$ and $P \in \mathcal{C}^2_\psi(\Omega)$ are such that
	\begin{equation} \label{eq forcing P bounds}
	a \leq P(\omega) \leq b \;\; \forall \omega \in \Omega.
	\end{equation}
	Consider the family $\{p_\omega\}_{\omega \in \Omega}$ of almost periodic forcing functions defined by
	\begin{equation} \label{eq a p forcing}
	p_\omega(t) = P(\omega + \psi(t)), \;\; t \in \RN.
	\end{equation}
	Let $v_{*}= 2\max \{ \max_{\varpi \in \Omega} \p_\psi P(\varpi),0\}$ and denote by
	\begin{equation*}
	E_\omega = \{ (t_0,v_0) \in \RN \times (v_*,\infty): (t_n,v_n)_{n\in\NN} \text{ is well defined and } \lim_{n\to\infty} v_n = \infty \}
	\end{equation*}
	the escaping set for the ping-pong map with forcing function $p(t) = 	p_\omega(t)$. Then, for almost all $\omega \in \Omega$, the set $E_\omega \subset \RN^2$ has Lebesgue measure zero.
\end{thm} 
\begin{rmk} \label{rmk v stern}
The notation  $v_{*}= 2\max \{ \max_{\varpi \in \Omega} \p_\psi P(\varpi),0\}$ is consistent with \eqref{eq def v stern}, since for every $\omega \in \Omega$ the set $\omega \cdot \RN$ lies dense in $\Omega$ and thus
\begin{equation*}
\sup_{t\in\RN} \dot{p}_\omega(t) = \sup_{t\in\RN}  \p_\psi P(\omega + \psi(t)) = \max_{\varpi \in \Omega} \p_\psi P(\varpi).
\end{equation*}
\end{rmk}
We will give some further preliminaries before starting the actual proof. First we note, that the ping-pong map $(t_0,v_0) \mapsto (t_1,v_1)$ is not symplectic. To remedy this defect, we reformulate the model in terms of time $t$ and energy $E=\frac{1}{2}v^2$. In these new coordinates the ping-pong map becomes 
\begin{align} \label{eq def ping pong map P}
\mathcal{P}:(&t_0,E_0) \mapsto (t_1,E_1), \\
&t_1 = \tilde{t} + \frac{p(\tilde{t})}{\sqrt{2 E_1}}, \;\; E_1 = E_0 - 2\sqrt{2 E_0}\dot{p}(\tilde{t}) + 2\dot{p}(\tilde{t})^2 = (\sqrt{E_0}-\sqrt{2}\dot{p}(\tilde{t}))^2,
\end{align}
where $\tilde{t}=\tilde{t}(t_0,E_0)$ is determined implicitly through the relation $\tilde{t} = t_0 + \frac{p(\tilde{t})}{\sqrt{2E_0}}$. This map is defined for $(t_0,E_0) \in \RN \times (\frac{1}{2}v_*^2,\infty)$. Since it has a generating function \cite[Lemma 3.7]{kunze_ortega_complete_orbits}, it is measure-preserving. Furthermore, from the inverse function theorem we can derive that $\mathcal{P}$ is locally injective. Note however, that in general $\mathcal{P}$ fails to be injective globally (see Appendix \ref{sec appendix ping pong}).\\
Now, we will demonstrate that $W(t_0,E_0)= p(t_0)^2 E_0$ acts as an adiabatic invariant for the ping-pong map. For this purpose we will cite the following lemma \cite[Lemma 5.1]{kunze_ortega_complete_orbits}:
\begin{lem} \label{lem invariante delta t und E}
	There is a constant $C>0$, depending only upon $\lVert p \rVert_{\mathcal{C}^2}$ and $a,b>0$ from \eqref{general forcing function}, such that
	\begin{equation*}
	\lvert p(t_1)^2 E_1 - p(t_0)^2 E_0\rvert \leq C \Delta(t_0,E_0) \;\; \forall (t_0,E_0)\in \RN\times (v_*^2/2,\infty),
	\end{equation*}
	where $(t_1,E_1)=\mathcal{P}(t_0,E_0)$ denotes the ping-pong map for the forcing $p$, and $\Delta(t_0,E_0) = E_0^{-1/2} + \sup\{ \lvert \ddot{p}(t) - \ddot{p}(s)\rvert : t,s \in [t_0-C,t_0+C],\lvert t-s \rvert \leq C E_0^{-1/2} \}$.
\end{lem}
So far we have depicted the case of a general forcing function $p$. Now we will replace $p(t)$ by $p_\omega(t)$ from \eqref{eq a p forcing} and study the resulting ping-pong map. First we note that due to $P \in \mathcal{C}^2_\psi(\Omega)$ we have $p_\omega \in \mathcal{C}^2(\RN)$. Also $0<a\leq p_\omega(t) \leq b$ holds for all $\omega \in \Omega$ by assumption. Furthermore, since $\omega \cdot \RN$ lies dense in $\Omega$ it is
\begin{equation*}
\lVert {p}_\omega \rVert_\infty =  \lVert P \rVert_\infty, \;\;
\lVert \dot{p}_\omega \rVert_\infty =  \lVert \p_\psi P \rVert_\infty, \;\;
\lVert \ddot{p}_\omega \rVert_\infty =  \lVert \p_\psi^2 P \rVert_\infty.
\end{equation*}
In particular this means $\lVert p_\omega \rVert_{\mathcal{C}^2(\RN)} = \lVert P \rVert_{\mathcal{C}^2_\psi(\Omega)}$ for all $\omega \in \Omega$.  Therefore all considerations above apply with uniform constants. As depicted in Remark \ref{rmk v stern}, also the threshold $v_*= 2\max \{ \max_{\varpi \in \Omega} \p_\psi P(\varpi),0\}$ is uniform in $\omega$. Finally, since $\ddot{p}_\omega(t) = \p_\psi^2 P(\omega +\psi(t))$, the function $\Delta(t_0,E_0)$ can be uniformly bounded by
\begin{equation*}
\Delta(E_0) = E_0^{-1/2} + \sup\{ \lvert \p_\psi^2 P(\varpi) - \p_\psi^2 P(\varpi') \rvert : \varpi,\varpi'\in \Omega, \lVert \varpi - \varpi' \rVert \leq C E_0^{-1/2}  \}.
\end{equation*}
Hence, from Lemma \ref{lem invariante delta t und E} we obtain
\begin{lem} \label{lem ping pong adiabatic invariant}
	There is a constant $C>0$, uniform in $\omega \in \Omega$, such that
	\begin{equation*}
	\lvert p(t_1)^2 E_1 - p(t_0)^2 E_0\rvert \leq C \Delta(E_0) \;\; \forall (t_0,E_0)\in \RN\times (v_*^2/2,\infty),
	\end{equation*}
	where $(t_0,E_0)\mapsto(t_1,E_1)$ denotes the ping-pong map $\mathcal{P}$ for the forcing function $p_\omega(t)$.
\end{lem}
Consider the equation
\begin{equation} \label{def tau}
\tau = \frac{1}{\sqrt{2E_0}}P(\omega_0 + \psi(\tau)).
\end{equation}
Since $P \in \mathcal{C}^1_\psi(\Omega)$ and $1-(2E_0)^{-1/2} \p_\psi P(\omega_0 + \psi(\tau)) \geq \frac{1}{2}>0$ for $E_0 > \frac{1}{2}v_*^2$, equation (\ref{def tau}) can be solved implicitly for $\tau=\tau(\omega_0,E_0)\in \mathcal{C}(\Omega \times (v_*^2/2,\infty))$ (cf. \cite{Biasi} for a suitable implicit function theorem). For $\omega \in \Omega$ and $t_0\in \RN$ one can consider (\ref{def tau}) with $\omega_0 = \omega + \psi(t_0)$. Then $P\in \mathcal{C}^1_\psi(\Omega)$ and the classical implicit function theorem yield $\tau \in \mathcal{C}^1_\psi(\Omega\times(v_*^2/2,\infty))$. Moreover, comparing this to the definition of $\tilde{t}$, we observe the following relation:
\begin{equation} \label{rel tilde t tau}
\tilde{t}(t_0,E_0) = t_0 + \tau(\omega + \psi(t_0), E_0).
\end{equation}
Now we will give the proof of the main theorem, in which we will link the ping-pong map corresponding to $p_\omega(t)$ to the setup of Section \ref{sec a theorem about escaping sets}.

\begin{proof}[Proof of Theorem \ref{thm main theorem almost periodic}]
Let $\mathcal{D} = \Omega \times (E^*,\infty)$, where $E^* = \max\{ \frac{1}{2}v_*^2,E_{**} \}$ and $E_{**}$ will be determined below. Consider $f:\mathcal{D}\subset \Omega \times (0,\infty) \to \Omega \times (0,\infty), f(\omega_0,E_0) = (\omega_1,E_1)$, given by
\begin{equation*}
\omega_1 = \omega_0 + \psi(F(\omega_0,E_0)), \;\; E_1 = E_0 + G(\omega_0,E_0),
\end{equation*}
where
\begin{align*}
&F(\omega_0,E_0) = \bigg(\frac{1}{\sqrt{2E_0}} + \frac{1}{\sqrt{2E_1}}\bigg) P(\omega_0 + \psi(\tau)),\\
&G(\omega_0,E_0) = -2\sqrt{2E_0} \p_\psi P(\omega_0 + \psi(\tau)) + 2\p_\psi P(\omega_0 + \psi(\tau))^2,
\end{align*}
for $\tau = \tau(\omega_0,E_0)$. Then $f$ has special form \eqref{embedding form} and therefore we can study the family $\{ f_\omega \}_{\omega \in \Omega}$ of planar maps defined by \eqref{planar maps}. But plugging \eqref{rel tilde t tau} into the definition of $\mathcal{P}$ shows, that $f_\omega$ is just the ping-pong map $\mathcal{P}$ in the case of the forcing $p_\omega(t)$. Independently of $\omega$, these maps are defined on $D_\omega = (\psi_\omega\times \id)^{-1}(\mathcal{D}) = \RN\times (E^*,\infty)$.

Let us show that $f$ is injective on $\Omega\times(E_{**},\infty)$, if $E_{**}$ is sufficiently large. Therefore suppose $f(\omega_0,E_0) = (\omega_1,E_1) = f(\tilde{\omega}_0,\tilde{E}_0)$. Since $\omega_0 + \iota(F(\omega_0,E_0)) = \tilde{\omega}_0 + \iota(F(\tilde{\omega}_0,\tilde{E}_0))$ there is $\omega\in \Omega$ and $t_0,\tilde{t}_0\in \RN$ such that $\omega_0 = \omega +\psi(t_0)$ and $\tilde{\omega}_0 = \omega +\psi(\tilde{t}_0)$. Implicit differentiation yields $\p_{t_0} \tau(\omega + \psi(t_0),E_0) = \mathcal{O}(E_0^{-1/2})$ and $\p_{E_0} \tau(\omega + \psi(t_0),E_0) = \mathcal{O}(E_0^{-3/2})$. Moreover, $E_1=\mathcal{O}(E_0)$ implies
\begin{equation*}
D_{f_\omega}(t_0,E_0) = \begin{pmatrix}
1 + \mathcal{O}(E_0^{-1/2}) & \mathcal{O}(E_0^{-3/2}) \\
\mathcal{O}(E_0^{1/2}) & 1 + \mathcal{O}(E_0^{-1/2})
\end{pmatrix}
\end{equation*}
for the Jacobian matrix of $f_\omega$. Throughout this paragraph $C$ will denote positive constants depending on $E_{**}$ and $\lVert P \rVert_{\mathcal{C}^2_\psi(\Omega)}$, which will not be further specified. Without loss of generality we may assume $E_0 \leq \tilde{E}_0$. Then, applying the mean value theorem yields $\lvert t_0 - \tilde{t}_0 \rvert \leq C E_0^{-1/2} \lvert t_0 - \tilde{t}_0 \rvert + C E_0^{-3/2} \lvert E_0 - \tilde{E}_0 \rvert$ and $\lvert E_0 - \tilde{E}_0 \rvert \leq C \tilde{E}_0^{1/2} \lvert t_0 - \tilde{t}_0 \rvert   + C E_0^{-1/2} \lvert E_0 - \tilde{E}_0 \rvert$, provided $E_{**}$ is sufficiently big. Thus, for large $E_{**}$ we get $\lvert t_0 - \tilde{t}_0 \rvert \leq  C E_0^{-3/2} \lvert E_0 - \tilde{E}_0 \rvert$ and $\lvert E_0 - \tilde{E}_0 \rvert \leq C \tilde{E}_0^{1/2} \lvert t_0 - \tilde{t}_0 \rvert$. Now, combining these inequalities gives us $\lvert t_0 - \tilde{t}_0 \rvert \leq C E_0^{-3/2} \tilde{E}_0^{1/2} \lvert t_0 - \tilde{t}_0 \rvert$. But since $E_1 =\mathcal{O}(E_0)$ and also $\tilde{E}_0 =\mathcal{O}(E_1)$, we can conclude $\lvert t_0 - \tilde{t}_0 \rvert \leq C E_0^{-1}  \lvert t_0 - \tilde{t}_0 \rvert$. In turn, this implies $t_0 = \tilde{t}_0$ and $E_0 = \tilde{E}_0$ for $E_{**}$ sufficiently large, which proves the injectivity of $f_\omega$ and $f$.

Next we want to show that $f$ is also measure-preserving. To this end, consider the maps $g:\Sigma\times [0,S) \times (E^*,\infty) \to \Sigma\times [0,\infty) \times (0,\infty)$ defined by
\begin{equation*}
g(\sigma,s,E) = (\sigma, f_{\sigma}(s,E))
\end{equation*}
and $\chi: \Sigma\times [0,\infty)  \to \Sigma\times [0,S) , \;\; \chi(\sigma,t) = \Phi^{-1}(\sigma\cdot t)$ from \eqref{eq def chi}. Then, the identity
\begin{equation*}
f = (\Phi\times \id) \circ (\chi\times \id) \circ g \circ (\Phi^{-1}\times \id)
\end{equation*}
 holds on $\mathcal{D}$. This can be illustrated as follows:
\begin{equation*}
\begin{tikzcd}
(\omega_0,E_0) \arrow[r, mapsto, "f"] \arrow[d, mapsto, "\Phi^{-1}\times \id"']
& (\omega_1,E_1)  \\
(\sigma_0,s_0,E_0) \arrow[r, mapsto, "g"]  
& (\sigma_0,s_1,E_1)  \arrow[r, mapsto, "\chi\times \id"]
& (\sigma_1,s_1',E_1) \arrow[lu, mapsto, "\Phi \times \id"']
\end{tikzcd}
\end{equation*}
Recalling Lemma \ref{lem decomp of haar measure} and the fact that $f_\omega$ has a generating function, it suffices to show that $\chi\times\id$ preserves the measure of any Borel set $\mathcal{B} \subset g \left( (\Phi^{-1}\times \id)(\mathcal{D})\right)$. Therefore, consider the sets 
\begin{equation*}
\mathcal{B}_k = \mathcal{B} \cap \left(\Sigma\times [(k-1)S,kS) \times (0,\infty)\right), \;\; k\in\NN.
\end{equation*}
Then we have
\begin{equation*}
(\mu_{\Sigma}\otimes \lambda^2) \left((\chi\times \id)(\mathcal{B}_k )\right) = (\mu_{\Sigma}\otimes \lambda^2) \left(\mathcal{B}_k\right),
\end{equation*}
as depicted in Section \ref{sec decomp haar measure}. Moreover, the injectivity of $f$ implies the injectivity of $\chi\times \id$ on $\mathcal{B}$ and thus the sets $(\chi\times \id)(\mathcal{B}_k)$ are mutually disjoint. Since $\mathcal{B}=\cup_{k\in\NN} \mathcal{B}_k$, this yields $(\mu_{\Sigma}\otimes \lambda^2) \left((\chi\times \id)(\mathcal{B})\right) =  (\mu_{\Sigma}\otimes \lambda^2) \left(\mathcal{B}\right)$.

Finally, we need to find a function $W\in \mathcal{C}^1_\psi(\Omega\times (0,\infty))$ such that \eqref{W growth} and \eqref{ineq adiabatic inv} are verified. For this define
\begin{equation*}
W(\omega_0,E_0) = P(\omega_0)^2 E_0.
\end{equation*}
Conditions \eqref{W growth} clearly holds if we take $\beta=a^2$ and $\delta = b^2$ with $a,b$ from \eqref{eq forcing P bounds}. Moreover, the definition of $f$ yields
\begin{align*}
W(f(\omega_0,E_0)) - W(\omega_0,E_0) &= P(\omega_1)^2 E_1 - P(\omega_0)^2 E_0 \\
&= P(\omega_0 + \psi(F(\omega_0,E_0)))^2 E_1 - P(\omega_0)^2 E_0 \\
&= p_{\omega_0}(F(\omega_0,E_0))^2 E_1 - p_{\omega_0}(0)^2 E_0.
\end{align*}
Now let $t_0=0$ and $(t_1,E_1)= f_{\omega_0}(t_0,E_0)$. Then $t_1 = F(\omega_0,E_0)$ and thus Lemma \ref{lem ping pong adiabatic invariant} yields
\begin{equation*}
W(f(\omega_0,E_0)) - W(\omega_0,E_0) = p_{\omega_0}(t_1)^2 E_1 - p_{\omega_0}(t_0)^2 E_0 \leq C \Delta(E_0),
\end{equation*}
where $C>0$ is uniform in $\omega_0$. But then taking $k(E_0) = C\Delta(E_0)$ proves \eqref{ineq adiabatic inv}, since $\lim_{r\to\infty}\Delta(r) = 0$ follows from $\p_\psi^2P \in \mathcal{C}(\Omega)$. 

Now we have validated all conditions of Theorem \ref{thm escaping set nullmenge} for the map $f:\mathcal{D} \to \Omega\times (0,\infty)$. Applying it yields $\lambda^2(\hat{E}_\omega)=0$ for almost all $\omega \in \Omega$, where $\hat{E}_\omega = \{ (t_0,E_0) \in \hat{D}_{\omega,\infty}: \lim_{n\to\infty} E_n = \infty \}$ and $\hat{D}_{\omega,\infty}$ is defined as in Section \ref{sec a p successor maps}. This can be translated back to the original coordinates $(t,v) = (t,\sqrt{2E})$: Let us denote by $g_\omega$ the ping-pong map $(t_0,v_0)\mapsto (t_1,v_1)$ from \eqref{eq ping pong map t v} for the forcing $p(t)=p_\omega(t)$ and let
\begin{equation*}
\tilde{D}_\omega = \RN \times (\sqrt{2E^*},\infty), \;\;\; \tilde{D}_{\omega,1} = \tilde{D}_\omega, \;\;\; \tilde{D}_{\omega,n+1} = g_\omega^{-1}(\tilde{D}_{\omega,n}), \;\;\; \tilde{D}_{\omega,\infty} = \bigcap_{n=1}^{\infty} \tilde{D}_{\omega,n}.
\end{equation*}
Then $\lambda^2(\tilde{E}_\omega) = 0$ for almost all $\omega \in \Omega$, where $\tilde{E}_\omega = \{ (t_0,v_0) \in \tilde{D}_{\omega,\infty}: \lim_{n\to\infty} v_n = \infty \}$. Now, consider the escaping set $E_\omega$ from the theorem and take $(t_0,v_0) \in E_\omega$. Since $\lim_{n\to\infty} v_n = \infty$, there is $n_0\in \NN$ such that $v_n > \sqrt{2E^*}$ for all $n\geq n_0$. But this just means $(t_n,v_n) \in \tilde{E}_\omega$ for $n\geq n_0$. In particular, this implies $E_\omega \subset \bigcup_{n\in\NN} g_\omega^{-n}(\tilde{E}_\omega)$. Considering that $g_\omega$ is area-preserving, this proves the assertion: $\lambda^2(E_\omega)=0$ for almost all $\omega \in \Omega$.
\end{proof}

\begin{rmk}
	Let us also point out that the framework developed in the present paper can be applied to a lot of other dynamical systems. A famous example of such a system is given by the so called Littlewood boundedness problem. There, the question is whether solutions of an equation $\ddot{x} + G'(x) = p(t)$ stay bounded in the $(x,\dot{x})$-phase space if the potential $G$ satisfies some superlinearity condition. In \cite{schliessauf_littlewood} it is shown that the associated escaping set $E$ typically has Lebesgue measure zero for $G'(x) = \lvert x \rvert^{\alpha-1}x$ with $\alpha \geq 3$ and a quasi-periodic forcing function $p(t)$. Indeed, this result can be improved to the almost periodic case in a way analogous to the one presented here (for the ping-pong problem).
\end{rmk}

\section{Appendix}

\subsection{The hull of an almost periodic function}

\begin{lem} \label{lem uniform limit}
	Let $u\in \mathcal{C}(\RN)$ be almost periodic. If the sequences $\{ u_{\tau_n} \}, \{ u_{s_n} \}$ are uniformly convergent, then $\{ u_{\tau_n-s_n} \}$ is uniformly convergent as well. 
\end{lem}

\begin{proof}
	Let $\epsilon>0$ be given. Since $\{ u_{\tau_n} \}, \{ u_{s_n} \}$ are Cauchy sequences, there exists $N\in\NN$ such that for $n,m\geq N$ we have
	\begin{equation*}
	\lvert u_{\tau_n}(-s_n+t) - u_{\tau_m}(-s_n+t) \rvert < \frac{\epsilon}{2}
	\end{equation*}
	and
	\begin{equation*}
	\lvert u_{s_n}(\tau_m -s_n - s_m+t) - u_{s_m}(\tau_m -s_n - s_m +t) \rvert < \frac{\epsilon}{2},
	\end{equation*}
	where $t\in \RN$ is arbitrary. Together this yields
	\begin{equation*}
	\lvert u(\tau_n -s_n +t) - u(\tau_m - s_m +t) \rvert < \epsilon.
	\end{equation*}
	for all $n,m\geq N$ and $t\in\RN$, and thus proves the assertion.
\end{proof}

\subsection{Ping-pong map} \label{sec appendix ping pong}
The map $\mathcal{P}$ from \eqref{eq def ping pong map P} can fail to be injective globally. For this, suppose there are $\tilde{t}_1,\tilde{t}_2\in \RN$ with $\tilde{t}_1<\tilde{t}_2$ such that the derivative $\dot{p}(t)$ reaches its maximum at both $\tilde{t}_1$ and $\tilde{t}_2$, and moreover $p(\tilde{t}_1)>p(\tilde{t}_2)$. For the sake of simplicity, let us consider the original coordinates $(t,v)$. Let $v_1>0$ be the unique number so that $\tilde{t}_1 + \frac{p(\tilde{t}_1)}{v_1} = \tilde{t}_2 + \frac{p(\tilde{t}_2)}{v_1}$. Now, we define $v_0 = v_1 + 2\dot{p}(\tilde{t}_1) = v_1 + 2\dot{p}(\tilde{t}_2)$ and $t_{0,i} = \tilde{t}_i - \frac{p(\tilde{t}_i)}{v_1}$ for $i=1,2$. From $p(\tilde{t}_1)>p(\tilde{t}_2)$ we can derive $t_{0,1}<t_{0,2}$. But $v_0 = v_1 + 2\sup_{t\in\RN}\dot{p}(t) > v_*$ implies that $(t_{0,i},v_0)$ are in the domain of $\mathcal{P}$ and furthermore $\mathcal{P}(t_{0,i},v_0) = (t_1,v_1)$, where $t_1 = \tilde{t}_1 + \frac{p(\tilde{t}_1)}{v_1}$.

\bibliographystyle{alpha}
\bibliography{bib2}

\end{document}